\documentclass[12pt,twoside]{amsart}
\usepackage[latin1]{inputenc}
\usepackage{amsmath, amsthm, amscd, amsfonts, amssymb, graphicx}
\usepackage[bookmarksnumbered, plainpages]{hyperref}

\newtheorem{thm}{Theorem}[section]
\newtheorem{cor}[thm]{Corollary}

\numberwithin{equation}{section}

\begin{document}

\title{{\bf $SL(2,{\bf Z})$ modular forms and anomaly cancellation formulas II}}

\author{ Jianyun Guan; Yong Wang*\\
 }

\date{}

\thanks{{\scriptsize
\hskip -0.4 true cm \textit{2010 Mathematics Subject Classification:}
58C20; 57R20; 53C80.
\newline \textit{* Corresponding author.}}}

\maketitle

\begin{abstract}
 By some $SL(2,{\bf Z})$ modular forms introduced in \cite{Li2} and \cite{CHZ}, we construct some $\Gamma^0(2)$ and $\Gamma_0(2)$ modular forms and obtain some new cancellation formulas for spin manifolds and spin$^c$ manifolds respectively. As corollaries, we get some divisibility results of index of twisted Dirac operators on spin manifolds and spin$^c$ manifolds .
\end{abstract}

\vskip 0.2 true cm

\section{ Introduction}
In \cite{AW}, Alvarez-Gaum\'{e} and Witten discovered a formula that represents the beautiful relationship between the top
  components of the Hirzebruch $\widehat{L}$-form and $\widehat{A}$-form of a $12$-dimensional smooth Riemannian
  manifold. This formula is called the "miracle cancellation" formula for gravitational anomalies. In \cite{Li1}, Liu established higher dimensional "miraculous cancellation"
  formulas for $(8k+4)$-dimensional Riemannian manifolds by
  developing modular invariance properties of characteristic forms. These formulas could be used to deduce some divisibility results.
  In \cite{HZ1}, \cite{HZ2}, Han and Zhang established some more general cancellation formulas that involve a
  complex line bundle and gave their applications . In \cite{Wa}, Wang obtained some new anomaly cancellation formulas by studying the modular invariance of some characteristic forms. This formula was applied to spin manifolds and spin$^c$ manifolds, then  some results on divisibilities on spin
manifolds and spin$^c$ manifolds were derived. Moreover, Han, Liu
and Zhang using the Eisenstein series, a more general cancellation formula was derived \cite{HLZ1}. And in \cite{HLZ2}, the authors showed that both of the Green-Schwarz anomaly factorization formula for the gauge group $E_8\times E_8$ and the Horava-Witten anomaly factorization formula for the gauge group $E_8$ could be derived through modular forms of weight $14$. This answered a question of J. H. Schwarz. They also established generalizations of these decomposition formulas and obtained a new Horava-Witten type decomposition formula on $12$-dimensional manifolds. In \cite{HHLZ}, Han, Huang, Liu and Zhang introduced a modular form of weight $14$ over $SL(2,{\bf Z})$ and a modular form of weight $10$ over $SL(2,{\bf Z})$ and they got some interesting
anomaly cancellation formulas on $12$-dimensional manifolds. In \cite{Wa1} and \cite{Wa2}, Wang obtained some new anomaly cancellation formulas by studying some $SL(2,{\bf Z})$ modular forms. Some divisibility results of the index of the twisted Dirac operator are obtained. In \cite{Li2}, Liu introduced a modular form of a 2k-dimensional spin manifold with a weight of 2k. In \cite{CHZ}, Chen, Han and Zhang defined an integral modular form of weight $2k$ for a $4k$-dimensional $spin^c$ manifold and an integral modular form of weight $2k$ for a $4k+2$-dimensional $spin^c$ manifold. A natural question is whether there are more cancellation formulas and more results on divisibilities on spin manifold. Motivated by \cite{CHZ} and \cite{Wa2}, we introduce some modular forms of weight $2k$ over $\Gamma^0(2)$ and $\Gamma_0(2)$ by the $SL(2,{\bf Z})$ modular forms introduced in \cite{Li2} and \cite{CHZ} and derive some new anomaly cancellation formulas.  \\
\indent The structure of this paper is briefly described below: In Section 2, we have introduce some definitions and basic concepts that we will use in the paper. In Section 3, we prove the generalized cancellation formulas for even-dimensional spin manifolds. Finally, in section 4, we obtain some generalized cancellation formulas for even-dimensional spin$^c$ manifolds.\\

\section{Characteristic Forms and Modular Forms}
\quad The purpose of this section is to review the necessary knowledge on
characteristic forms and modular forms that we are going to use.\\

 \noindent {\bf  2.1 characteristic forms }\\
 \indent Let $M$ be a Riemannian manifold.
 Let $\nabla^{ TM}$ be the associated Levi-Civita connection on $TM$
 and $R^{TM}=(\nabla^{TM})^2$ be the curvature of $\nabla^{ TM}$.
 Let $\widehat{A}(TM,\nabla^{ TM})$ and $\widehat{L}(TM,\nabla^{ TM})$
 be the Hirzebruch characteristic forms defined respectively by (cf. \cite{Zh})
\begin{equation}
   \widehat{A}(TM,\nabla^{ TM})={\rm
det}^{\frac{1}{2}}\left(\frac{\frac{\sqrt{-1}}{4\pi}R^{TM}}{{\rm
sinh}(\frac{\sqrt{-1}}{4\pi}R^{TM})}\right),
\end{equation}
 \begin{equation}
     \widehat{L}(TM,\nabla^{ TM})={\rm
 det}^{\frac{1}{2}}\left(\frac{\frac{\sqrt{-1}}{2\pi}R^{TM}}{{\rm
 tanh}(\frac{\sqrt{-1}}{4\pi}R^{TM})}\right).
 \end{equation}
   Let $E$, $F$ be two Hermitian vector bundles over $M$ carrying
   Hermitian connection $\nabla^E,\nabla^F$ respectively. Let
   $R^E=(\nabla^E)^2$ (resp. $R^F=(\nabla^F)^2$) be the curvature of
   $\nabla^E$ (resp. $\nabla^F$). If we set the formal difference
   $G=E-F$, then $G$ carries an induced Hermitian connection
   $\nabla^G$ in an obvious sense. We define the associated Chern
   character form as
   \begin{equation}
       {\rm ch}(G,\nabla^G)={\rm tr}\left[{\rm
   exp}(\frac{\sqrt{-1}}{2\pi}R^E)\right]-{\rm tr}\left[{\rm
   exp}(\frac{\sqrt{-1}}{2\pi}R^F)\right].
   \end{equation}
   For any complex number $t$, let
   $$\wedge_t(E)={\bf C}|_M+tE+t^2\wedge^2(E)+\cdots,~S_t(E)={\bf
   C}|_M+tE+t^2S^2(E)+\cdots$$
   denote respectively the total exterior and symmetric powers of
   $E$, which live in $K(M)[[t]].$ The following relations between
   these operations hold,
   \begin{equation}
       S_t(E)=\frac{1}{\wedge_{-t}(E)},~\wedge_t(E-F)=\frac{\wedge_t(E)}{\wedge_t(F)}.
   \end{equation}
   Moreover, if $\{\omega_i\},\{\omega_j'\}$ are formal Chern roots
   for Hermitian vector bundles $E,F$ respectively, then
   \begin{equation}
       {\rm ch}(\wedge_t(E))=\prod_i(1+e^{\omega_i}t)
   \end{equation}
   Then we have the following formulas for Chern character forms,
   \begin{equation}
       {\rm ch}(S_t(E))=\frac{1}{\prod_i(1-e^{\omega_i}t)},~
{\rm ch}(\wedge_t(E-F))=\frac{\prod_i(1+e^{\omega_i}t)}{\prod_j(1+e^{\omega_j'}t)}.
   \end{equation}
\indent If $W$ is a real Euclidean vector bundle over $M$ carrying a
Euclidean connection $\nabla^W$, then its complexification $W_{\bf
C}=W\otimes {\bf C}$ is a complex vector bundle over $M$ carrying a
canonical induced Hermitian metric from that of $W$, as well as a
Hermitian connection $\nabla^{W_{\bf C}}$ induced from $\nabla^W$.
If $E$ is a vector bundle (complex or real) over $M$, set
$\widetilde{E}=E-{\rm dim}E$ in $K(M)$ or $KO(M)$.\\

\noindent{\bf 2.2 Some properties about the Jacobi theta functions
and modular forms}\\
   \indent We first recall the four Jacobi theta functions are
   defined as follows( cf. \cite{Ch}):
   \begin{equation}
      \theta(v,\tau)=2q^{\frac{1}{8}}{\rm sin}(\pi
   v)\prod_{j=1}^{\infty}[(1-q^j)(1-e^{2\pi\sqrt{-1}v}q^j)(1-e^{-2\pi\sqrt{-1}v}q^j)],
   \end{equation}
\begin{equation}
    \theta_1(v,\tau)=2q^{\frac{1}{8}}{\rm cos}(\pi
   v)\prod_{j=1}^{\infty}[(1-q^j)(1+e^{2\pi\sqrt{-1}v}q^j)(1+e^{-2\pi\sqrt{-1}v}q^j)],
\end{equation}
\begin{equation}
    \theta_2(v,\tau)=\prod_{j=1}^{\infty}[(1-q^j)(1-e^{2\pi\sqrt{-1}v}q^{j-\frac{1}{2}})
(1-e^{-2\pi\sqrt{-1}v}q^{j-\frac{1}{2}})],
\end{equation}
\begin{equation}
   \theta_3(v,\tau)=\prod_{j=1}^{\infty}[(1-q^j)(1+e^{2\pi\sqrt{-1}v}q^{j-\frac{1}{2}})
(1+e^{-2\pi\sqrt{-1}v}q^{j-\frac{1}{2}})],
\end{equation}
 \noindent
where $q=e^{2\pi\sqrt{-1}\tau}$ with $\tau\in\textbf{H}$, the upper
half complex plane. Let
\begin{equation}
    \theta'(0,\tau)=\frac{\partial\theta(v,\tau)}{\partial v}|_{v=0}.
\end{equation} \noindent Then the following Jacobi identity
(cf. \cite{Ch}) holds,
\begin{equation}   \theta'(0,\tau)=\pi\theta_1(0,\tau)\theta_2(0,\tau)\theta_3(0,\tau).
\end{equation}
\noindent Denote $$SL_2({\bf Z})=\left\{\left(\begin{array}{cc}
\ a & b  \\
 c  & d
\end{array}\right)\mid a,b,c,d \in {\bf Z},~ad-bc=1\right\}$$ the
modular group. Let $S=\left(\begin{array}{cc}
\ 0 & -1  \\
 1  & 0
\end{array}\right),~T=\left(\begin{array}{cc}
\ 1 &  1 \\
 0  & 1
\end{array}\right)$ be the two generators of $SL_2(\bf{Z})$. They
act on $\textbf{H}$ by $S\tau=-\frac{1}{\tau},~T\tau=\tau+1$. One
has the following transformation laws of theta functions under the
actions of $S$ and $T$ (cf. \cite{Ch}):
\begin{equation}
    \theta(v,\tau+1)=e^{\frac{\pi\sqrt{-1}}{4}}\theta(v,\tau),~~\theta(v,-\frac{1}{\tau})
=\frac{1}{\sqrt{-1}}\left(\frac{\tau}{\sqrt{-1}}\right)^{\frac{1}{2}}e^{\pi\sqrt{-1}\tau
v^2}\theta(\tau v,\tau);
\end{equation}
 \begin{equation}
     \theta_1(v,\tau+1)=e^{\frac{\pi\sqrt{-1}}{4}}\theta_1(v,\tau),~~\theta_1(v,-\frac{1}{\tau})
=\left(\frac{\tau}{\sqrt{-1}}\right)^{\frac{1}{2}}e^{\pi\sqrt{-1}\tau
v^2}\theta_2(\tau v,\tau);
 \end{equation}
\begin{equation}   \theta_2(v,\tau+1)=\theta_3(v,\tau),~~\theta_2(v,-\frac{1}{\tau})
=\left(\frac{\tau}{\sqrt{-1}}\right)^{\frac{1}{2}}e^{\pi\sqrt{-1}\tau
v^2}\theta_1(\tau v,\tau);
\end{equation}
\begin{equation}    \theta_3(v,\tau+1)=\theta_2(v,\tau),~~\theta_3(v,-\frac{1}{\tau})
=\left(\frac{\tau}{\sqrt{-1}}\right)^{\frac{1}{2}}e^{\pi\sqrt{-1}\tau
v^2}\theta_3(\tau v,\tau).
\end{equation}
\begin{equation}
    \theta'(v,\tau+1)=e^{\frac{\pi\sqrt{-1}}{4}}\theta'(v,\tau),~~
 \theta'(0,-\frac{1}{\tau})=\frac{1}{\sqrt{-1}}\left(\frac{\tau}{\sqrt{-1}}\right)^{\frac{1}{2}}
\tau\theta'(0,\tau),
\end{equation}

\noindent
 \noindent {\bf Definition 2.1} A modular form over $\Gamma$, a
 subgroup of $SL_2({\bf Z})$, is a holomorphic function $f(\tau)$ on
 $\textbf{H}$ such that
 \begin{equation}
    f(g\tau):=f\left(\frac{a\tau+b}{c\tau+d}\right)=\chi(g)(c\tau+d)^kf(\tau),
 ~~\forall g=\left(\begin{array}{cc}
\ a & b  \\
 c & d
\end{array}\right)\in\Gamma,
 \end{equation}
\noindent where $\chi:\Gamma\rightarrow {\bf C}^{\star}$ is a
character of $\Gamma$. $k$ is called the weight of $f$.\\
Let $$\Gamma_0(2)=\left\{\left(\begin{array}{cc}
\ a & b  \\
 c  & d
\end{array}\right)\in SL_2({\bf Z})\mid c\equiv 0~({\rm
mod}~2)\right\},$$
$$\Gamma^0(2)=\left\{\left(\begin{array}{cc}
\ a & b  \\
 c  & d
\end{array}\right)\in SL_2({\bf Z})\mid b\equiv 0~({\rm
mod}~2)\right\},$$
be the two modular subgroups of $SL_2({\bf Z})$.
It is known that the generators of $\Gamma_0(2)$ are $T,~ST^2ST$,
the generators of $\Gamma^0(2)$ are $STS,~T^2STS$ (cf. \cite{Ch}).\\
\indent If $\Gamma$ is a modular subgroup, let ${\mathcal{M}}_{{\bf
R}}(\Gamma)$ denote the ring of modular forms over $\Gamma$ with
real Fourier coefficients. Writing $\theta_j=\theta_j(0,\tau),~1\leq
j\leq 3,$ we introduce six explicit modular forms (cf. \cite{Li1}),
$$\delta_1(\tau)=\frac{1}{8}(\theta_2^4+\theta_3^4),~~\varepsilon_1(\tau)=\frac{1}{16}\theta_2^4\theta_3^4,$$
$$\delta_2(\tau)=-\frac{1}{8}(\theta_1^4+\theta_3^4),~~\varepsilon_2(\tau)=\frac{1}{16}\theta_1^4\theta_3^4,$$
\noindent They have the following Fourier expansions in
$q^{\frac{1}{2}}$:
$$\delta_1(\tau)=\frac{1}{4}+6q+\cdots,~~\varepsilon_1(\tau)=\frac{1}{16}-q+\cdots,$$
$$\delta_2(\tau)=-\frac{1}{8}-3q^{\frac{1}{2}}+\cdots,~~\varepsilon_2(\tau)=q^{\frac{1}{2}}+\cdots,$$
\noindent where the $"\cdots"$ terms are the higher degree terms,
all of which have integral coefficients. They also satisfy the
transformation laws,
\begin{equation}
    \delta_2(-\frac{1}{\tau})=\tau^2\delta_1(\tau),~~~~~~\varepsilon_2(-\frac{1}{\tau})
=\tau^4\varepsilon_1(\tau),
\end{equation}
\noindent {\bf Lemma 2.2} (\cite{Li1}) {\it $\delta_1(\tau)$ (resp.
$\varepsilon_1(\tau)$) is a modular form of weight $2$ (resp. $4$)
over $\Gamma_0(2)$, $\delta_2(\tau)$ (resp. $\varepsilon_2(\tau)$)
is a modular form of weight $2$ (resp. $4$) over $\Gamma^0(2)$,
while  $\delta_3(\tau)$ (resp. $\varepsilon_3(\tau)$) is a modular
form of weight $2$ (resp. $4$) over $\Gamma_\theta(2)$ and moreover
${\mathcal{M}}_{{\bf R}}(\Gamma^0(2))={\bf
R}[\delta_2(\tau),\varepsilon_2(\tau)]$.}

\section{The generalized cancellation formulas for even-dimensional spin manifolds}
Let $M$ be a $4k$-dimensional spin manifold and $\triangle(M)$ be the spinor bundle. Let $\widetilde{T_{\mathbf{C}}M}=T_{\mathbf{C}}M-\dim M$.
 Set
 \begin{equation}
   \Theta_1(T_{\mathbf{C}}M)=
   \bigotimes _{n=1}^{\infty}S_{q^n}(\widetilde{T_{\mathbf{C}}M})\otimes
\bigotimes _{m=1}^{\infty}\wedge_{q^m}(\widetilde{T_{\mathbf{C}}M})
,\end{equation}
\begin{equation}
\Theta_2(T_{\mathbf{C}}M)=\bigotimes _{n=1}^{\infty}S_{q^n}(\widetilde{T_{\mathbf{C}}M})\otimes
\bigotimes _{m=1}^{\infty}\wedge_{-q^{m-\frac{1}{2}}}(\widetilde{T_{\mathbf{C}}M}),
\end{equation}
\begin{equation}
\Theta_3(T_{\mathbf{C}}M)=\bigotimes _{n=1}^{\infty}S_{q^n}(\widetilde{T_{\mathbf{C}}M})\otimes
\bigotimes _{m=1}^{\infty}\wedge_{q^{m-\frac{1}{2}}}(\widetilde{T_{\mathbf{C}}M}).
\end{equation}
Let $V$ be a rank $2l$ real vector bundle on $M$. Moreover, $V_{\mathbf{C}}=V\otimes\mathbf{C}.$  Set
\begin{align}
Q_1(\tau)=&{\rm Ind}(D\otimes[\triangle(M)\otimes \Theta_1(T_{C}M)+2^{2k}\Theta_2(T_{C}M)\\\notag
&+2^{2k}\Theta_3(T_{C}M)]\otimes\triangle(V)\otimes\bigotimes _{n=1}^{\infty}\wedge_{q^{n}}(\widetilde {V_{\mathbf{C}}}))\\\notag
=&\int_M\widehat{A}(TM)[{\rm ch}(\triangle(M)){\rm ch}(\Theta_1(T_{C}M))+2^{2k}{\rm ch}(\Theta_2(T_{C}M))\\\notag
&+2^{2k}{\rm ch}(\Theta_3(T_{C}M))]
{\rm ch}(\triangle(V)\otimes\bigotimes _{n=1}^{\infty}\wedge_{q^{n}}(\widetilde {V_{\mathbf{C}}})).
\end{align}
\begin{align}
Q_2(\tau)
=&\int_M\widehat{A}(TM)[{\rm ch}(\triangle(M)){\rm ch}(\Theta_1(T_{C}M))+2^{2k}{\rm ch}(\Theta_2(T_{C}M))\\\notag
&+2^{2k}{\rm ch}(\Theta_3(T_{C}M))]
{\rm ch}(\bigotimes _{n=1}^{\infty}\wedge_{-q^{n-\frac{1}{2}}}(\widetilde {V_{\mathbf{C}}})).
\end{align}
\begin{align}
Q_3(\tau)
=&\int_M\widehat{A}(TM)[{\rm ch}(\triangle(M)){\rm ch}(\Theta_1(T_{C}M))+2^{2k}{\rm ch}(\Theta_2(T_{C}M))\\\notag
&+2^{2k}{\rm ch}(\Theta_3(T_{C}M))]
{\rm ch}(\bigotimes _{n=1}^{\infty}\wedge_{q^{n-\frac{1}{2}}}(\widetilde {V_{\mathbf{C}}})).
\end{align}
Let $\{\pm 2\pi iy_v\}$ be the formal Chern roots of $V_{\mathbf{C}}$. If $V$ is spin and $\triangle(V)$ is the associated spinor bundle of $V_{\mathbf{C}}$, one know that the Chern character of $\triangle(V)$ is given by
$${\rm ch}(\triangle(V))=\prod_{v=1}^l(e^{\pi iy_v}+e^{-\pi iy_v)}.$$
Set
$$\Theta_1(T_{C}M)\otimes\bigotimes _{n=1}^{\infty}\wedge_{-q^{n-\frac{1}{2}}}(\widetilde{V_{\mathbf{C}}})=B^1_0(T_{\mathbf{C}}M,V_{\mathbf{C}})+B^1_1(T_{\mathbf{C}}M,V_{\mathbf{C}})q^{\frac{1}{2}}+\cdots,$$
$$\Theta_2(T_{C}M)\otimes\bigotimes _{n=1}^{\infty}\wedge_{-q^{n-\frac{1}{2}}}(\widetilde{V_{\mathbf{C}}})=B^2_0(T_{\mathbf{C}}M,V_{\mathbf{C}})+B^2_1(T_{\mathbf{C}}M,V_{\mathbf{C}})q^{\frac{1}{2}}+\cdots,$$
$$\Theta_3(T_{C}M)\otimes\bigotimes _{n=1}^{\infty}\wedge_{-q^{n-\frac{1}{2}}}(\widetilde{V_{\mathbf{C}}})=B^3_0(T_{\mathbf{C}}M,V_{\mathbf{C}})+B^3_1(T_{\mathbf{C}}M,V_{\mathbf{C}})q^{\frac{1}{2}}+\cdots.$$
Let $p_1$ denote the first Pontryagin class. If $\omega$ is a differential form over $M$, we donete $\omega^{2k}$ its top degree component. Our main results include the following theorem.
\begin{thm}
 If $p_1(V)=0$, then
    \begin{equation}
        \{\widehat{A}(TM,\nabla^{TM})({\rm ch }(\triangle(M)){\rm ch }(\triangle(V))+2^{2k+1}{\rm ch }(\triangle(V)))\}^{4k}=2^{l+k}\Sigma_{r=0}^{[\frac{k}{2}]}2^{-6r}h_r.
    \end{equation}
where each $h_r, 0\leq r\leq [\frac{k}{2}],$ is a canonical integral linear combination of
\begin{equation}
\begin{split}
\{\widehat{A}&(TM){\rm ch}(\triangle(M)){\rm ch}(B^1_{\alpha}(T_{\mathbf{C}}M,V_{\mathbf{C}}))\\
&+2^{2k}\widehat{A}(TM){\rm ch}(B^2_{\alpha}(T_{\mathbf{C}}M,V_{\mathbf{C}}))+2^{2k}\widehat{A}(TM){\rm ch}(B^3_{\alpha}(T_{\mathbf{C}}M,V_{\mathbf{C}}))\}^{4k},\ \ \ 0\leq \alpha \leq j.\nonumber
\end{split}
\end{equation}
\end{thm}
\begin{proof}
    Let $\{\pm 2\pi\sqrt{-1}x_j|1\leq j\leq k\}$ be the Chern roots of $T_{\mathbf{C}}M.$ We have
    \begin{equation}
        \begin{split}
           Q_1(\tau)=&\int_M\widehat{A}(TM){\rm ch}(\triangle(M)){\rm ch}(\triangle(V))+2^{2k+1}\widehat{A}(TM){\rm ch}(\triangle(V))\\
           &+[\widehat{A}(TM){\rm ch}(\triangle(M)){\rm ch}(\triangle(V)){\rm ch}(2\widetilde{T_{\mathbf{C}}M}+\widetilde{V_{\mathbf{C}}})\\
         &+2^{2k+1}\widehat{A}(TM){\rm ch}(\triangle(V)){\rm ch}(\widetilde{T_{\mathbf{C}}M}+\wedge^2\widetilde{T_{\mathbf{C}}M}+\widetilde{V_{\mathbf{C}}})]q+O(q^{2})
        \end{split}
    \end{equation}
    \begin{equation}
    \begin{split}
         Q_2(\tau)=&\int_M\widehat{A}(TM){\rm ch}(\triangle(M))+2^{2k+1}\widehat{A}(TM)-[\widehat{A}(TM){\rm ch}(\triangle(M)){\rm ch}(\widetilde{T_{\mathbf{C}}M})\\
         &+2^{2k+1}\widehat{A}(TM){\rm ch}(\widetilde{T_{\mathbf{C}}M})]q^{\frac{1}{2}}+O(q)\\
         =&\int_M\widehat{A}(TM){\rm ch}(\triangle(M)){\rm ch}(B^1_0(T_{\mathbf{C}}M,V_{\mathbf{C}}))+2^{2k}\widehat{A}(TM){\rm ch}(B^2_0(T_{\mathbf{C}}M,V_{\mathbf{C}}))\\
         &+2^{2k}\widehat{A}(TM){\rm ch}(B^3_0(T_{\mathbf{C}}M,V_{\mathbf{C}}))
         +[\widehat{A}(TM){\rm ch}(\triangle(M)){\rm ch}(B^1_1(T_{\mathbf{C}}M,V_{\mathbf{C}}))\\
         &+2^{2k}\widehat{A}(TM){\rm ch}(B^2_1(T_{\mathbf{C}}M,V_{\mathbf{C}}))+2^{2k}\widehat{A}(TM){\rm ch}(B^3_1(T_{\mathbf{C}}M,V_{\mathbf{C}}))]q^{\frac{1}{2}}+\cdots.
    \end{split}
    \end{equation}
    Moreover, we can direct computations show that
    \begin{equation}
    {\widehat{A}(TM,\nabla^{TM})}{\rm ch}(\triangle(M)){\rm ch}(\Theta_1(T_{C}M))=\prod_{j=1}^{2k}\frac{2 x_j\theta'(0,\tau)}{\theta(x_j,\tau)}\frac{\theta_1(x_j,\tau)}{\theta_1(0,\tau)},
    \end{equation}
    \begin{equation}{\widehat{A}(TM,\nabla^{TM})}{\rm ch}(2^{2k}\Theta_2(T_{C}M))=\prod_{j=1}^{2k}\frac{2 x_j\theta'(0,\tau)}{\theta(x_j,\tau)}\frac{\theta_2(x_j,\tau)}{\theta_2(0,\tau)},\end{equation}
 \begin{equation}{\widehat{A}(TM,\nabla^{TM})}{\rm ch}(2^{2k}\Theta_3(T_{C}M))=\prod_{j=1}^{2k}\frac{2 x_j\theta'(0,\tau)}{\theta(x_j,\tau)}\frac{\theta_3(x_j,\tau)}{\theta_3(0,\tau)},\end{equation}
 \begin{equation}
     {\rm ch}(\triangle(V)\otimes\bigotimes_{n=1}^{\infty}\wedge_{q^n }(\widetilde{V_{\mathbf{C}}}))=2^l\prod_{v=1}^{l}\frac{\theta_1(y_v,\tau)}{\theta_1(0,\tau)},
 \end{equation}
 \begin{equation}
    {\rm ch}(\bigotimes_{n=1}^{\infty}\wedge_{-q^{n-\frac{1}{2}}}(\widetilde{V_{\mathbf{C}}}))=\prod_{v=1}^{l}\frac{\theta_2(y_v,\tau)}{\theta_2(0,\tau)},
 \end{equation}
 \begin{equation}
    {\rm ch}(\bigotimes_{n=1}^{\infty}\wedge_{q^{n-\frac{1}{2}}}(\widetilde{V_{\mathbf{C}}}))=\prod_{v=1}^{l}\frac{\theta_3(y_v,\tau)}{\theta_3(0,\tau)}.
 \end{equation}
So we have
\begin{equation}Q_1(\tau)=\prod_{j=1}^{2k}\frac{2x_j\theta'(0,\tau)}{\theta(x_j,\tau)}
\left(\prod_{j=1}^{2k}\frac{\theta_1(x_j,\tau)}{\theta_1(0,\tau)}+\prod_{j=1}^{2k}\frac{\theta_2(x_j,\tau)}{\theta_2(0,\tau)}+\prod_{j=1}^{2k}\frac{\theta_3(x_j,\tau)}
{\theta_3(0,\tau)}\right)\cdot2^l\prod_{v=1}^{l}\frac{\theta_1(y_v,\tau)}{\theta_1(0,\tau)}.\end{equation}
Similarly, we have
\begin{equation}Q_2(\tau)=\prod_{j=1}^{2k}\frac{2x_j\theta'(0,\tau)}{\theta(x_j,\tau)}
\left(\prod_{j=1}^{2k}\frac{\theta_1(x_j,\tau)}{\theta_1(0,\tau)}+\prod_{j=1}^{2k}\frac{\theta_2(x_j,\tau)}{\theta_2(0,\tau)}+\prod_{j=1}^{2k}\frac{\theta_3(x_j,\tau)}
{\theta_3(0,\tau)}\right)\cdot \prod_{v=1}^{l}\frac{\theta_2(y_v,\tau)}{\theta_2(0,\tau)},\end{equation}
\begin{equation}Q_3(\tau)=\prod_{j=1}^{2k}\frac{2x_j\theta'(0,\tau)}{\theta(x_j,\tau)}
\left(\prod_{j=1}^{2k}\frac{\theta_1(x_j,\tau)}{\theta_1(0,\tau)}+\prod_{j=1}^{2k}\frac{\theta_2(x_j,\tau)}{\theta_2(0,\tau)}+\prod_{j=1}^{2k}\frac{\theta_3(x_j,\tau)}
{\theta_3(0,\tau)}\right)\cdot \prod_{v=1}^{l}\frac{\theta_3(y_v,\tau)}{\theta_3(0,\tau)}.\end{equation}
Let $P_1(\tau)=Q_1(\tau)^{4k},$ $P_2(\tau)=Q_2(\tau)^{4k},$ $P_3(\tau)=Q_3(\tau)^{4k}.$
By $(2.13)-(2.17)$ and $p_1(V)=0,$ then $P_1(\tau)$ is a modular form of weight $2k$ over $\Gamma_0(2)$, where $P_2(\tau)$ is a modular form of weight $2k$ over $\Gamma^0(2).$ Moreover, the following identity holds:
\begin{equation}
    P_1(-\frac{1}{\tau})=2^l\tau^{2k}P_2(\tau).
\end{equation}
Observe that at any point $x\in M,$ up to the volume form determined by the metric on $T_xM,$ both $P_i(\tau),~i=1,2,$ can be viewed as a power series of $q^{\frac{1}{2}}$ with real Fourier coefficients. By Lemma 2.2, we have
\begin{equation}
P_2(\tau)=h_0(8\delta_2)^{k}+h_1(8\delta_2)^{k-2}\varepsilon_2+\cdots+h_{[\frac{k}{2}]}(8\delta_2)^{k-2}
\varepsilon^{[\frac{k}{2}]}_2,
\end{equation}
where each $h_r,~0\leq r\leq [\frac{k}{2}],$ is a real multiple of the volume form at $x.$ By (2.19), (3.19) and (3.20), we get
\begin{equation}
P_1(\tau)=2^l[h_0(8\delta_1)^{k}+h_1(8\delta_1)^{k-2}\varepsilon_1+\cdots+h_{[\frac{k}{2}]}(8\delta_1)^{k-2}
\varepsilon^{[\frac{k}{2}]}_1].
\end{equation}
By comparing the constant term in (3.21), we get (3.7). By comparing the coefficients of $q^{\frac{1}{2}},~j\geq 0$ between the two sides of (3.20), we can use the induction method to prove that each $h_r,~0\leq r\leq [\frac{k}{2}],$ can be expressed through a canonical integral linear combination of $\{\widehat{A}(TM){\rm ch}(\triangle(M)){\rm ch}(B^1_{\alpha}(T_{\mathbf{C}}M,V_{\mathbf{C}}))+2^{2k}\widehat{A}(TM){\rm ch}(B^2_{\alpha}(T_{\mathbf{C}}M,V_{\mathbf{C}}))+2^{2k}\widehat{A}(TM){\rm ch}(B^3_{\alpha}(T_{\mathbf{C}}M,V_{\mathbf{C}}))\}^{4k}.$ Here we write our the explicit expressions for $h_0$ and $h_1$ as follows:
\begin{align}
     h_0=&(-1)^k\{\widehat{A}(TM){\rm ch}(\triangle(M))+2^{2k+1}\widehat{A}(TM)\}^{4k},\\
     h_1=&(-1)^{k+1}\{(\widehat{A}(TM){\rm ch}(\triangle(M))+2^{2k+1}\widehat{A}(TM)){\rm ch}(\widetilde{V_{\mathbf{C}}}+24k)\}^{4k}.
\end{align}
The proof is completed.
\end{proof}
By comparing the coefficients of $q$ in (3.21), we have
\begin{thm}
    If $p_1(V)=0,$ then
    \begin{equation}
    \begin{split}
       \{&\widehat{A}(TM, \nabla^{TM}){\rm ch}(\triangle(M)){\rm ch}(\triangle(V)){\rm ch}(2\widetilde{T_{\mathbf{C}}M}+\widetilde{V_{\mathbf{C}}}-24k)\\
         &+2^{2k+1}\widehat{A}(TM){\rm ch}(\triangle(V)){\rm ch}(\widetilde{T_{\mathbf{C}}M}+\wedge^2\widetilde{T_{\mathbf{C}}M}+\widetilde{V_{\mathbf{C}}}-24k)\}^{4k}=-2^{l+k+6}\Sigma^{[\frac{k}{2}]}_{r=0}r2^{-6r}h_r,
    \end{split}
    \end{equation}
 where each $h_r, 0\leq r\leq [\frac{k}{2}],$ is a canonical integral linear combination of
\begin{equation}
\begin{split}
\{\widehat{A}&(TM){\rm ch}(\triangle(M)){\rm ch}(B^1_{\alpha}(T_{\mathbf{C}}M,V_{\mathbf{C}}))\\
&+2^{2k}\widehat{A}(TM){\rm ch}(B^2_{\alpha}(T_{\mathbf{C}}M,V_{\mathbf{C}}))+2^{2k}\widehat{A}(TM){\rm ch}(B^3_{\alpha}(T_{\mathbf{C}}M,V_{\mathbf{C}}))\}^{4k},\ \ \ 0\leq \alpha \leq j.\nonumber
\end{split}
\end{equation}
\end{thm}
Letting $k=2$, i.e. for 8-dimensional manifolds, we have
\begin{cor}
If $p_1(V)=0,$ then
\begin{equation}
\begin{split}
    \{\widehat{A}(&TM,\nabla^{TM})({\rm ch }(\triangle(M)){\rm ch }(\triangle(V))+32{\rm ch }(\triangle(V)))\}^{8}\\
    =&3\cdot2^{l-1}\{\widehat{A}(TM,\nabla^{TM}){\rm ch }(\triangle(M))+32\widehat{A}(TM,\nabla^{TM})\}^{8}\\
    &-2^{l-4}\{[\widehat{A}(TM,\nabla^{TM}){\rm ch }(\triangle(M))+32\widehat{A}(TM,\nabla^{TM})]{\rm ch}(T_{\mathbf{C}}M)\}^{8}
\end{split}
\end{equation}
\end{cor}
Letting $k=3$, i.e. for 12-dimensional manifolds, we have
\begin{cor}
If $p_1(V)=0,$ then
\begin{equation}
\begin{split}
    \{\widehat{A}(&TM,\nabla^{TM})({\rm ch }(\triangle(M)){\rm ch }(\triangle(V))+128{\rm ch }(\triangle(V)))\}^{12}\\
    =&-2^{l-1}\{\widehat{A}(TM,\nabla^{TM}){\rm ch }(\triangle(M))+128\widehat{A}(TM,\nabla^{TM})\}^{12}\\
    &+2^{l-3}\{[\widehat{A}(TM,\nabla^{TM}){\rm ch }(\triangle(M))+128\widehat{A}(TM,\nabla^{TM})]{\rm ch}(T_{\mathbf{C}}M)\}^{12}
\end{split}
\end{equation}
\end{cor}
By the Atiyah-Singer index theorem for spin manifolds, we have
\begin{cor}
    If $M$ and $V$ are spin and $p_1(V)=0,$ then
    \begin{equation}
\begin{split}
    {\rm Ind}(D\otimes(\triangle(M)\otimes\triangle(V)+2^{2k+1}\triangle(V)))=2^{l+k}\Sigma_{r=0}^{[\frac{k}{2}]}2^{-6r}h_r,
\end{split}
    \end{equation}
where each $h_r, 0\leq r\leq [\frac{k}{2}],$ is a canonical integral linear combination of
${\rm Ind}(D\otimes(\triangle(M)\otimes B^1_{\alpha}(T_{\mathbf{C}}M,V_{\mathbf{C}})+2^{2k}\cdot B^2_{\alpha}(T_{\mathbf{C}}M,V_{\mathbf{C}})+ 2^{2k}\cdot B^3_{\alpha}(T_{\mathbf{C}}M,V_{\mathbf{C}}))),\ 0\leq \alpha \leq j.$
\end{cor}
Let $k=2m+1,$ then each $h_j$ is an even integer, so we have
\begin{cor}
    If $M$ and $V$ are spin and $p_1(V)=0$ with $l\geq 4m+2,$ then ${\rm Ind}(D\otimes(\triangle(M)\otimes\triangle(V)+2^{2k+1}\triangle(V)))=0~({\rm mod}~16).$
\end{cor}
Let $M$ be a $(8m+4)-$dimensional spin manifold with boundary $N.$ Let
\begin{equation}
\begin{split}
h_j=\{&\widehat{A}(TM){\rm ch}(\triangle(M)){\rm ch}(B^1_j(T_{\mathbf{C}}M,V_{\mathbf{C}}))\\
&+2^{2k}\widehat{A}(TM){\rm ch}(B^2_j(T_{\mathbf{C}}M,V_{\mathbf{C}}))+2^{2k}\widehat{A}(TM){\rm ch}(B^3_j(T_{\mathbf{C}}M,V_{\mathbf{C}}))\}^{4k}\nonumber
\end{split}
\end{equation}
and let $\widetilde{\eta}(a)$ and $\widetilde{\eta}(B_j,b)$ denote the reduced eta invariants associated with the operators $D\otimes(\triangle(M)\otimes\triangle(V)+2^{2k+1}\triangle(V))$ and $D\otimes(\triangle(M)\otimes B^1_j+2^{2k}\cdot B^2_j+ 2^{2k}\cdot B^3_j)$ on $N$ respectively (cf. \cite{APS}). Then by the Atiyah-Patodi-Singer index theorem, we have
\begin{cor}
    If $M$ and $V$ are spin and $p_1(V)=0$ with $l\geq 4m+2,$ then
    \begin{equation}
\begin{split}
    {\rm Ind}(D\otimes(\triangle(M)\otimes\triangle(V)+2^{2k+1}\triangle(V)))=-\widetilde{\eta}(a)+2^{l+k}\Sigma_{r=0}^{[\frac{k}{2}]}2^{-6r}\widetilde{\eta}(B_j,b)~~(mod~16).
\end{split}
    \end{equation}
\end{cor}
By Theorem 3.2, we have
\begin{cor}
    If $M$ and $V$ are spin and $p_1(V)=0$ with $l\geq 4m+2,$ then
    \begin{equation}
    \begin{split}
       \{&\widehat{A}(TM, \nabla^{TM}){\rm ch}(\triangle(M)){\rm ch}(\triangle(V)){\rm ch}(2\widetilde{T_{\mathbf{C}}M}+\widetilde{V_{\mathbf{C}}}-24(2m+1))\\
         &+2^{4m+3}\widehat{A}(TM){\rm ch}(\triangle(V)){\rm ch}(\widetilde{T_{\mathbf{C}}M}+\wedge^2\widetilde{T_{\mathbf{C}}M}+\widetilde{V_{\mathbf{C}}}-24(2m+1))\}^{8m+4}
    \end{split}
    \end{equation}
    is divisible by $2^9.$
\end{cor}

\section{Some generalized cancellation formulas involving a complex line
bundle for even-dimensional spin$^c$ manifolds
}
Let $M$ be closed oriented ${\rm spin^{c}}$-manifold and $L$ be the complex line bundle associated to the given ${\rm spin^{c}}$ structure on $M.$ Denote by $c=c_1(L)$ the first Chern class of $L.$ Also, we use $L_{\bf{R}}$ for the notation of $L,$ when it is viewed as an oriented real plane bundle.
Let $\Theta(T_{\mathbf{C}}M,L_{\bf{R}}\otimes\bf{C})$ be the virtual complex vector bundle over $M$ defined by
\begin{equation}
    \begin{split}
        \Theta(T_{\mathbf{C}}M,L_{\bf{R}}\otimes\mathbf{C})=&\bigotimes _{n=1}^{\infty}S_{q^n}(\widetilde{T_{\mathbf{C}}M})\otimes
\bigotimes _{m=1}^{\infty}\wedge_{q^m}(\widetilde{L_{\bf{R}}\otimes\mathbf{C}})\\
&\otimes
\bigotimes _{r=1}^{\infty}\wedge_{-q^{r-\frac{1}{2}}}(\widetilde{L_{\bf{R}}\otimes\mathbf{C}})\otimes
\bigotimes _{s=1}^{\infty}\wedge_{q^{s-\frac{1}{2}}}(\widetilde{L_{\bf{R}}\otimes\mathbf{C}}).\nonumber
    \end{split}
\end{equation}
Let ${\rm dim}M=4k$ and $u=-\frac{\sqrt{-1}}{2\pi}c.$ Set
\begin{equation}
    \widetilde{Q}_1(\tau)=\left\{\widehat{A}(TM,\nabla^{TM}){\rm exp}(\frac{c}{2}){\rm ch}(\triangle(V)){\rm ch}\left[\Theta(T_{\mathbf{C}}M,L_{\bf{R}}\otimes\mathbf{C})\otimes\bigotimes _{n=1}^{\infty}\wedge_{q^{n}}(\widetilde{V_{\mathbf{C}}})\right]\right\},
\end{equation}
\begin{equation}
    \widetilde{Q}_2(\tau)=\left\{\widehat{A}(TM,\nabla^{TM}){\rm exp}(\frac{c}{2}){\rm ch}\left[\Theta(T_{\mathbf{C}}M,L_{\bf{R}}\otimes\mathbf{C})\otimes\bigotimes _{n=1}^{\infty}\wedge_{-q^{n-\frac{1}{2}}}(\widetilde{V_{\mathbf{C}}})\right]\right\}.
\end{equation}
Moreover, $\Theta(T_{\mathbf{C}}M,L_{\bf{R}}\otimes\mathbf{C})\otimes\bigotimes _{n=1}^{\infty}\wedge_{q^{n}}(\widetilde{V_{\mathbf{C}}})$ and $\Theta(T_{\mathbf{C}}M,L_{\bf{R}}\otimes\mathbf{C})\otimes\bigotimes _{n=1}^{\infty}\wedge_{-q^{n-\frac{1}{2}}}(\widetilde{V_{\mathbf{C}}})$ admit formal Fourier expansion in $q^{\frac{1}{2}}$ as
\begin{equation}
\begin{split}
    \Theta(T_{\mathbf{C}}M,&L_{\bf{R}}\otimes\mathbf{C})\otimes\bigotimes _{n=1}^{\infty}\wedge_{q^{n}}(\widetilde{V_{\mathbf{C}}})\\
&=\widetilde{A}_0(T_{\mathbf{C}}M,L_{\bf{R}}\otimes\mathbf{C},V_{\mathbf{C}})+\widetilde{A}_1(T_{\mathbf{C}}M,L_{\bf{R}}\otimes\mathbf{C}, V_{\mathbf{C}})q^{\frac{1}{2}}+\cdots,\nonumber
\end{split}
\end{equation}
\begin{equation}
\begin{split}
    \Theta(T_{\mathbf{C}}M,&L_{\bf{R}}\otimes\mathbf{C})\otimes\bigotimes _{n=1}^{\infty}\wedge_{-q^{n-\frac{1}{2}}}(\widetilde{V_{\mathbf{C}}})\\
&=\widetilde{B}_0(T_{\mathbf{C}}M,L_{\bf{R}}\otimes\mathbf{C},V_{\mathbf{C}})+\widetilde{B}_1(T_{\mathbf{C}}M,L_{\bf{R}}\otimes\mathbf{C}, V_{\mathbf{C}})q^{\frac{1}{2}}+\cdots.\nonumber
\end{split}
\end{equation}
Similarly (3.16), we can direct computations show that
\begin{equation}
\widetilde{Q}_1(\tau)=\left(\prod_{j=1}^{2k}\frac{x_j\theta'(0,\tau)}{\theta(x_j,\tau)}
\left(\frac{\theta_1(u,\tau)}{\theta_1(0,\tau)}\frac{\theta_2(u,\tau)}{\theta_2(0,\tau)}
\frac{\theta_3(u,\tau)}{\theta_3(0,\tau)}
\right)\cdot 2^l\prod_{v=1}^{l}\frac{\theta_1(y_v,\tau)}{\theta_1(0,\tau)}\right)^{(4k)}.
\end{equation}
\begin{equation}
\widetilde{Q}_2(\tau)=\left(\prod_{j=1}^{2k}\frac{x_j\theta'(0,\tau)}{\theta(x_j,\tau)}
\left(\frac{\theta_1(u,\tau)}{\theta_1(0,\tau)}\frac{\theta_2(u,\tau)}{\theta_2(0,\tau)}
\frac{\theta_3(u,\tau)}{\theta_3(0,\tau)}
\right)\cdot \prod_{v=1}^{l}\frac{\theta_2(y_v,\tau)}{\theta_2(0,\tau)}\right)^{(4k)}.
\end{equation}
Let $\widetilde{P}_1(\tau)=\widetilde{Q}_1(\tau)^{4k},$ $\widetilde{P}_2(\tau)=\widetilde{Q}_2(\tau)^{4k}.$
By $(2.13)-(2.17)$ and $3p_1(L_{\mathbf{R}})+p_1(V)-p_1(M)=0,$ then $\widetilde{P}_1(\tau)$ is a modular form of weight $2k$ over $\Gamma_0(2)$, where $\widetilde{P}_2(\tau)$ is a modular form of weight $2k$ over $\Gamma^0(2).$ Moreover, the following identity holds:
\begin{equation}
    \widetilde{P}_1(-\frac{1}{\tau})=2^l\tau^{2k}\widetilde{P}_2(\tau).
\end{equation}
Playing the same game as in the proof of Theorem 3.1, we obtain
\begin{thm}
    If $3p_1(L_{\mathbf{R}})+p_1(V)-p_1(M)=0,$ then
    \begin{equation}
        \left\{\widehat{A}(TM,\nabla^{TM}){\rm exp}(\frac{c}{2}){\rm ch}(\triangle(V))\right\}^{4k}=2^{l+k}\Sigma_{r=0}^{[\frac{k}{2}]}2^{-6r}h_r.
    \end{equation}
where each $h_r, 0\leq r\leq [\frac{k}{2}],$ is a canonical integral linear combination of
$$\{\widehat{A}(TM,\nabla^{TM}){\rm exp}(\frac{c}{2}){\rm ch }(\widetilde{B}_\alpha(T_{\mathbf{C}}M,L_{\bf{R}}\otimes\mathbf{C},V_{\mathbf{C}}))\}^{4k},\ \ \ 0\leq \alpha \leq j.$$
\end{thm}
Similarly, comparing the coefficients of $q$ in (3.21), we have
\begin{thm}
    If $3p_1(L_{\mathbf{R}})+p_1(V)-p_1(M)=0,$ then
    \begin{equation}
    \begin{split}
        \{\widehat{A}&(TM,\nabla^{TM}){\rm exp}(\frac{c}{2}){\rm ch}(\triangle(V)){\rm ch}[(\widetilde{T_{\mathbf{C}}M}+\widetilde{L_{\bf{R}}\otimes\mathbf{C}}\\
    &+2\wedge^2\widetilde{L_{\bf{R}}\otimes\mathbf{C}}-(\widetilde{L_{\bf{R}}\otimes\mathbf{C}})\otimes(\widetilde{L_{\bf{R}}\otimes\mathbf{C}})+\widetilde{V_{\mathbf{C}}})
    -24k]\}^{4k}=-2^{l+k+6}\Sigma^{[\frac{k}{2}]}_{r=0}r2^{-6r}h_r.
    \end{split}
    \end{equation}
where each $h_r, 0\leq r\leq [\frac{k}{2}],$ is a canonical integral linear combination of
$$\{\widehat{A}(TM,\nabla^{TM}){\rm exp}(\frac{c}{2}){\rm ch }(\widetilde{B}_\alpha(T_{\mathbf{C}}M,L_{\bf{R}}\otimes\mathbf{C},V_{\mathbf{C}}))\}^{4k},\ \ \ 0\leq \alpha \leq j.$$
\end{thm}
By the Atiyah-Singer index theorem for spin$^c$ manifolds, we have
\begin{cor}
   If $M$ is spin$^c$ and $V$ is spin and $3p_1(L_{\mathbf{R}})+p_1(V)-p_1(M)=0,$ then
   \begin{equation}
        {\rm Ind}(D^c\otimes(\triangle(V)))=2^{l+k}\Sigma_{r=0}^{[\frac{k}{2}]}2^{-6r}h_r.
    \end{equation}
where each $h_r, 0\leq r\leq [\frac{k}{2}],$ is a canonical integral linear combination of
${\rm Ind}(D^c\otimes\widetilde{B}_\alpha(T_{\mathbf{C}}M,L_{\bf{R}}\otimes\mathbf{C},V_{\mathbf{C}})),~0\leq \alpha \leq j,$ and $D^c$ is the spin$^c$ Dirac operator.
\end{cor}
Let $k=2m+1,$ then each $h_j$ is an even integer, we have
\begin{cor}
   If $M$ is spin$^c$ and $V$ is spin and $3p_1(L_{\mathbf{R}})+p_1(V)-p_1(M)=0$ with $l\geq4m+2,$ then ${\rm Ind}(D^c\otimes(\triangle(V)))=0~({\rm mod}~16).$
\end{cor}
\begin{cor}
If $M$ is spin$^c$ and $V$ is spin and $3p_1(L_{\mathbf{R}})+p_1(V)-p_1(M)=0$ with $l\geq4m+2,$ then
    \begin{equation}
    \begin{split}
        \{\widehat{A}&(TM,\nabla^{TM}){\rm exp}(\frac{c}{2}){\rm ch}(\triangle(V)){\rm ch}[(\widetilde{T_{\mathbf{C}}M}+\widetilde{L_{\bf{R}}\otimes\mathbf{C}}\\  &+2\wedge^2\widetilde{L_{\bf{R}}\otimes\mathbf{C}}-(\widetilde{L_{\bf{R}}\otimes\mathbf{C}})\otimes(\widetilde{L_{\bf{R}}\otimes\mathbf{C}})
        +\widetilde{V_{\mathbf{C}}})-24(2m+1)]\}^{8m+4}\nonumber
    \end{split}
    \end{equation}
is divisible by $2^9.$
\end{cor}
Let ${\rm dim}M=4k+2$ and $u=-\frac{\sqrt{-1}}{2\pi}c.$ Let
\begin{equation}
    \Bar{Q}_1(\tau)=\left\{\widehat{A}(TM,\nabla^{TM}){\rm exp}(\frac{c}{2}){\rm ch}(\triangle(V)){\rm ch}\left[\Theta^*(T_{\mathbf{C}}M,L_{\bf{R}}\otimes\mathbf{C})\otimes\bigotimes _{n=1}^{\infty}\wedge_{q^{n}}(\widetilde{V_{\mathbf{C}}})\right]\right\},
\end{equation}
\begin{equation}
    \Bar{Q}_2(\tau)=\left\{\widehat{A}(TM,\nabla^{TM}){\rm exp}(\frac{c}{2}){\rm ch}\left[\Theta^*(T_{\mathbf{C}}M,L_{\bf{R}}\otimes\mathbf{C})\otimes\bigotimes _{n=1}^{\infty}\wedge_{-q^{n-\frac{1}{2}}}(\widetilde{V_{\mathbf{C}}})\right]\right\},
\end{equation}
where
$$\Theta^*(T_{\mathbf{C}}M,L_{\bf{R}}\otimes\mathbf{C})=\bigotimes _{n=1}^{\infty}S_{q^n}(\widetilde{T_{\mathbf{C}}M})\otimes
\bigotimes _{m=1}^{\infty}\wedge_{-q^m}(\widetilde{L_{\bf{R}}\otimes\mathbf{C}}),$$
and $\Theta^*(T_{\mathbf{C}}M,L_{\bf{R}}\otimes\mathbf{C})\otimes\bigotimes _{n=1}^{\infty}\wedge_{q^{n}}(\widetilde{V_{\mathbf{C}}})$ and $\Theta^*(T_{\mathbf{C}}M,L_{\bf{R}}\otimes\mathbf{C})\otimes\bigotimes _{n=1}^{\infty}\wedge_{-q^{n-\frac{1}{2}}}(\widetilde{V_{\mathbf{C}}})$ admit formal Fourier expansion in $q^{\frac{1}{2}}$ as
\begin{equation}
\begin{split}
    \Theta^*(T_{\mathbf{C}}M,&L_{\bf{R}}\otimes\mathbf{C})\otimes\bigotimes _{n=1}^{\infty}\wedge_{q^{n}}(\widetilde{V_{\mathbf{C}}})\\
&=\Bar{A}_0(T_{\mathbf{C}}M,L_{\bf{R}}\otimes\mathbf{C},V_{\mathbf{C}})+\Bar{A}_1(T_{\mathbf{C}}M,L_{\bf{R}}\otimes\mathbf{C}, V_{\mathbf{C}})q^{\frac{1}{2}}+\cdots,\nonumber
\end{split}
\end{equation}
\begin{equation}
\begin{split}
    \Theta^*(T_{\mathbf{C}}M,&L_{\bf{R}}\otimes\mathbf{C})\otimes\bigotimes _{n=1}^{\infty}\wedge_{-q^{n-\frac{1}{2}}}(\widetilde{V_{\mathbf{C}}})\\
&=\Bar{B}_0(T_{\mathbf{C}}M,L_{\bf{R}}\otimes\mathbf{C},V_{\mathbf{C}})+\Bar{B}_1(T_{\mathbf{C}}M,L_{\bf{R}}\otimes\mathbf{C}, V_{\mathbf{C}})q^{\frac{1}{2}}+\cdots.\nonumber
\end{split}
\end{equation}
Then, we have
\begin{equation}
\Bar{Q}_1(\tau)=\left\{\left(\prod_{j=1}^{2k+1}\frac{x_j\theta'(0,\tau)}{\theta(x_j,\tau)}\right)
\frac{\sqrt{-1}\theta(u,\tau)}{\theta_1(0,\tau)\theta_2(0,\tau)
\theta_3(0,\tau)}\cdot 2^l\prod_{v=1}^{l}\frac{\theta_1(y_v,\tau)}{\theta_1(0,\tau)}
\right\}^{(4k+2)},
\end{equation}
\begin{equation}
\Bar{Q}_2(\tau)=\left\{\left(\prod_{j=1}^{2k+1}\frac{x_j\theta'(0,\tau)}{\theta(x_j,\tau)}\right)
\frac{\sqrt{-1}\theta(u,\tau)}{\theta_1(0,\tau)\theta_2(0,\tau)
\theta_3(0,\tau)}\cdot \prod_{v=1}^{l}\frac{\theta_2(y_v,\tau)}{\theta_2(0,\tau)}
\right\}^{(4k+2)}.
\end{equation}
Let $\Bar{P}_1(\tau)=\Bar{Q}_1(\tau)^{4k+2},$ $\Bar{P}_2(\tau)=\Bar{Q}_2(\tau)^{4k+2}.$
By $(2.13)-(2.17)$ and $p_1(L_{\mathbf{R}})+p_1(V_{\mathbf{C}})-p_1(M)=0,$ then $\Bar{P}_1(\tau)$ is a modular form of weight $2k$ over $\Gamma_0(2)$, where $\Bar{P}_2(\tau)$ is a modular form of weight $2k$ over $\Gamma^0(2).$ Moreover, the following identity holds:
\begin{equation}
    \Bar{P}_1(-\frac{1}{\tau})=2^l\tau^{2k}\Bar{P}_2(\tau).
\end{equation}
Playing the same game as in the proof of Theorem 3.1, we obtain
\begin{thm}
   Let ${\rm dim}M=4k+2$ and $p_1(L_{\mathbf{R}})+p_1(V)-p_1(M)=0,$ then
    \begin{equation}
        \left\{\widehat{A}(TM,\nabla^{TM}){\rm exp}(\frac{c}{2}){\rm ch}(\triangle(V))\right\}^{4k+2}=2^{l+k}\Sigma_{r=0}^{[\frac{k}{2}]}2^{-6r}h_r.
    \end{equation}
    where each $h_r, 0\leq r\leq [\frac{k}{2}],$ is a canonical integral linear combination of
$$\{\widehat{A}(TM,\nabla^{TM}){\rm exp}(\frac{c}{2}){\rm ch }(\Bar{B}_\alpha(T_{\mathbf{C}}M,L_{\bf{R}}\otimes\mathbf{C},V_{\mathbf{C}}))\}^{4k+2},\ \ \ 0\leq \alpha \leq j.$$
\end{thm}
By the Atiyah-Singer index theorem for spin$^c$ manifolds, we have
\begin{cor}
   If $M$ is spin$^c$ and $V$ is spin and $p_1(L_{\mathbf{R}})+p_1(V)-p_1(M)=0,$ then
   \begin{equation}
        {\rm Ind}(D^c\otimes(\triangle(V)))=2^{l+k}\Sigma_{r=0}^{[\frac{k}{2}]}2^{-6r}h_r.
    \end{equation}
where each $h_r, 0\leq r\leq [\frac{k}{2}],$ is a canonical integral linear combination of
${\rm Ind}(D^c\otimes\Bar{B}_\alpha(T_{\mathbf{C}}M,L_{\bf{R}}\otimes\mathbf{C},V_{\mathbf{C}})),~0\leq \alpha \leq j.$
\end{cor}
Similarly, comparing the coefficients of $q$ in (3.21), we have
\begin{thm}
     Let ${\rm dim}M=4k+2$ and $p_1(L_{\mathbf{R}})+p_1(V)-p_1(M)=0,$ then
    \begin{equation}
    \begin{split}
        \{\widehat{A}(TM,\nabla^{TM}){\rm exp}(\frac{c}{2}){\rm ch}(\triangle(V)){\rm ch}[(\widetilde{T_{\mathbf{C}}M}-\widetilde{L_{\bf{R}}\otimes\mathbf{C}}+\widetilde{V_{\mathbf{C}}})-24k]\}^{4k+2}=-2^{l+k+6}\Sigma^{[\frac{k}{2}]}_{r=0}r2^{-6r}h_r.
    \end{split}
    \end{equation}
    where each $h_r, 0\leq r\leq [\frac{k}{2}],$ is a canonical integral linear combination of
$$\{\widehat{A}(TM,\nabla^{TM}){\rm exp}(\frac{c}{2}){\rm ch }(\Bar{B}_\alpha(T_{\mathbf{C}}M,L_{\bf{R}}\otimes\mathbf{C},V_{\mathbf{C}}))\}^{4k+2},\ \ \ 0\leq \alpha \leq j.$$
\end{thm}
Let $k=2m+1,$ then each $h_j$ is an even integer, we have
\begin{cor}
   If $M$ is spin$^c$ and $V$ is spin and $p_1(L_{\mathbf{R}})+p_1(V)-p_1(M)=0$ with $l\geq4m+2,$ then ${\rm Ind}(D^c\otimes(\triangle(V)))=0~({\rm mod}~32).$
\end{cor}
\begin{cor}
   If $M$ is spin$^c$ and $V$ is spin and $p_1(L_{\mathbf{R}})+p_1(V)-p_1(M)=0$ with $l\geq4m+2,$ then
    \begin{equation}
    \begin{split}
        \{\widehat{A}(TM,\nabla^{TM}){\rm exp}(\frac{c}{2}){\rm ch}(\triangle(V)){\rm ch}[(\widetilde{T_{\mathbf{C}}M}-\widetilde{L_{\bf{R}}\otimes\mathbf{C}}+\widetilde{V_{\mathbf{C}}})-24(2m+1)]\}^{8m+6}\nonumber
    \end{split}
    \end{equation}
is divisible by $2^{10}.$
\end{cor}
\section{Acknowledgements}

 The author was supported in part by NSFC No.11771070. The author also thank the referee for his (or her) careful reading and helpful comments.

\vskip 1 true cm

\section{Data availability}

No data was gathered for this article.

\section{Conflict of interest}

The authors have no relevant financial or non-financial interests to disclose.

\vskip 1 true cm

\bigskip
\bigskip
\indent{J. Guan}\\
 \indent{School of Mathematics and Statistics,
Northeast Normal University, Changchun Jilin, 130024, China }\\
\indent E-mail: {\it guanjy@nenu.edu.cn }\\
\indent{Y. Wang}\\
 \indent{School of Mathematics and Statistics,
Northeast Normal University, Changchun Jilin, 130024, China }\\
\indent E-mail: {\it wangy581@nenu.edu.cn }\\

\end{document}